\documentclass[12pt]{amsart}

\usepackage[all,color]{xy}
\usepackage{pb-diagram}
\usepackage[mathscr]{eucal}
\usepackage{hyperref}
\hypersetup{
    colorlinks=true, 
    linktoc=all,     
    linkcolor=blue,  
}

\usepackage{xcolor}

\usepackage[toc,page]{appendix}
\usepackage{pifont}
\usepackage{combelow} 

\RequirePackage{ifthen,setspace,enumitem,tikz}
\usetikzlibrary{arrows}

\DeclareMathAlphabet{\mathpzc}{OT1}{pzc}{m}{it}

\usepackage{amsfonts}
\usepackage{amsmath}
\usepackage{amssymb}

\newcommand{\tr}{\textnormal{tr}}
\newcommand{\ric}{\textnormal{Ric}}

\newcommand{\dbar}{\overline{\partial}}

\newcommand{\ddt}[1]{\frac{\partial #1}{\partial t}}
\newcommand{\R}[4]{R_{#1 \overline{#2} #3 \overline{#4}}}

\newcommand{\diam}{\textnormal{Diam}}
\newcommand{\RCD}{\textnormal{RCD}}
\newcommand{\dgh}{d_{\textnormal{GH}}}

\newcommand{\ddbar}{\sqrt{-1}\partial\dbar}

\def\v{{\vskip .1in}}

\def\cB{{\mathcal B}}

\def\cF{{\mathcal F}}

\def\cK{\mathcal{K}}

\def\R{{\mathbb R}}

\def\PSH{\textnormal{PSH}}
\def\Vol{\textnormal{Vol}}

\newtheorem{claim}{Claim}
\newtheorem{theorem}{Theorem}[section]

\newtheorem{lemma}{Lemma}[section]
\newtheorem{definition}{Definition}[section]
\newtheorem{corollary}{Corollary}[section]

\newtheorem{conjecture}{Conjecture}[section]
\newtheorem{question}{Question}[section]

\numberwithin{equation}{section}

\begin{document}

\title{Geometric stability for complex Monge-Amp\`ere equations}

\author{Bin Guo $^*$, Jian Song $^\dagger$ and Jacob Sturm $^{\dagger \dagger}$ }

\address{$^*$ Department of Mathematics \& Computer Science, Rutgers University, Newark, NJ 07102}

\email{bguo@rutgers.edu}

\address{$^{\dagger}$ Department of Mathematics, Rutgers University, Piscataway, NJ 08854}

\email{jiansong@math.rutgers.edu}

\address{$^{\dagger\dagger}$  Department of Mathematics \& Computer Science, Rutgers University, Newark, NJ 07102}

\email{sturm@newark.rutgers.edu}

\thanks{Research supported in part by the National Science Foundation under grants DMS-2203607 and DMS-2505575.}

\begin{abstract} Let $X$ be a compact  K\"ahler manifold. The analytic stability theorem of Ko\l odziej for complex Monge-Amp\`ere equation states that for any  K\"ahler metrics $\omega$ and $\omega'$ in the same cohomology class, if their volume measures are bounded in $L^p(X)$ (for some $p>1$) and close in $L^1(X)$, then their  K\"ahler potentials are close in $L^\infty(X)$. In this paper, we establish the geometric stability for complex Monge-Amp\`ere equations that $L^1$-closeness of volume measures implies $L^\infty$-closeness for the induced distance functions by $\omega$ and $\omega'$. Consequently, we prove that any non-smooth  K\"ahler current with volume measure bounded in $L^p$ (for some $p>1$) and  Ricci current bounded below induces a unique metric space, which turns out to be a compact RCD space homeomorphic to $X$ itself.

{\footnotesize }

\end{abstract}

\maketitle

\section{Introduction}

Let $(X, \theta)$ be an $n$-dimensional compact  K\"ahler manifold equipped with a fixed smooth  K\"ahler metric $\theta$. We define \begin{equation}\label{pKspace}
\cK_\theta(p, K)
\end{equation}
to be the set of all smooth  K\"ahler metrics $\omega\in [\theta]$ satisfying
\begin{equation}\label{lp}
    \left\|\frac{\omega^n}{\theta^n }\right\|_{L^p(X, \theta^n)}\leq K
    \end{equation}
with $p>1$. For convenience, we always assume $p>1$ throughout the paper.
If we let 
$$\omega = \theta + \ddbar \varphi, ~f= -\log \frac{\omega^n}{\theta^n}$$ 
with $\sup_X \varphi =0$, then $\varphi$ satisfies the following complex Monge-Amp\`ere equation
\begin{equation}\label{maeqn1}
(\theta+ \ddbar \varphi)^n = e^{-f}\theta^n, ~\sup_X \varphi=0. 
\end{equation}

The classical $L^\infty$-estimate of Ko\l odziej \cite{Kol1} shows that  $\|\varphi\|_{L^\infty(X)}$ is uniformly bounded for fixed $p$ and $K$.  Ko\l odziej \cite{Kol2} further establishes the analytic stability for equation (\ref{maeqn1}). For any $\epsilon>0$ and any two  K\"ahler metrics $$\omega=\theta+\ddbar\varphi, ~\omega'=\theta+\ddbar \varphi' \in \cK_\theta(p, K), ~\sup_X\varphi = \sup_X \varphi'=0,$$
there exists $\delta=\delta(p, K, \epsilon)>0$ such that if 
\begin{equation}\label{l1sm}
\| \omega^n - (\omega')^n \|_{L^1(X)} < \delta, 
\end{equation}
then $$\|\varphi- \varphi'\|_{L^\infty(X)} < \epsilon. 
$$
This stability result immediately implies the uniqueness for solutions of complex Monge-Amp\`ere equations with $e^{-f}\in L^p(X, \theta^n)$.

The norm $\|(\omega')^n-\omega^n\|_{L^1(X)}$ in (\ref{l1close}) is  the total variation distance between the volume measures $\omega^n$ and $(\omega')^n$, which  can also be expressed as  
$$\| (\omega')^n - \omega^n\|_{L^1(X)} = \left\| \frac{(\omega')^n}{\omega^n} - 1\right\|_{L^1(X, \omega^n)} = \left\|\frac{(\omega')^n}{\theta^n}- \frac{\omega^n}{\theta^n}\right\|_{L^1(X, \theta^n)}.$$

The main goal of this paper is to establish geometric stability for the complex Monge-Amp\`ere equation (\ref{maeqn1}) in terms of the distance functions induced by $\omega$ and $\omega'$ instead of the  K\"ahler potentials. It has always been a challenge to transform the analytic estimates into geometric estimates in the theory of complex Monge-Amp\`ere equations until recent progress made in \cite{GPSS1, GPSS2}. 

\medskip

\subsection{Main results.} ~Let us now state our main result.

\begin{theorem}\label{thm:main1} Let $(X, \theta)$ be a compact K\"ahler manifold of complex dimension $n$. For any $\epsilon>0$ and any $\Lambda\geq 0$, there exists
$\delta=\delta(X,\theta,p,K,\Lambda,\epsilon)>0$ such that for any $\omega\in \cK_\theta(p,K)$ with Ricci curvature bounded below by $-\Lambda$ and 
$\omega'\in \cK_\theta(p, K)$ satisfying
\begin{equation}\label{l1close}
\left\| (\omega')^n  -  \omega^n  \right\|_{L^1(X)} < \delta,
\end{equation}
we have
\begin{equation}\label{stab1}
d_{\omega'}(x, y) < d_{\omega}(x, y) + \epsilon
\end{equation}
for all $x,y\in X$, where $d_\omega$ and $d_{\omega'}$ are distance functions on $X$ induced by $\omega$ and $\omega'$.

\end{theorem}

The estimate (\ref{stab1}) can be viewed as semi-stability for the distance in the sense that the distance function $d_\omega$ is upper semi-continuous in the space $\cK_\theta(p, K)$ with respect to $L^1(X)$. Theorem \ref{thm:main1} is both natural and surprising. One certainly does not expect any second order estimate to control $\omega'$ from above if only assuming $L^p$ and $L^1$ bounds on $(\omega')^n$. On the other hand, it is always an intriguing problem to connect  K\"ahler metrics or currents to distance functions in a quantitative way without many techniques available. One of the key ingredients in the proof of Theorem \ref{thm:main1} is the segment inequality developed in \cite{CC1}, which turns out to be a bridging technique between geometric distance and plurisubharmonic functions.

To improve the semi-stability Theorem \ref{thm:main1}, we will have to impose a uniform Ricci curvature lower bound on the space of $\cK_\theta(p,K)$.

\begin{corollary}\label{cor:main1} For any $\epsilon>0$ and $\Lambda>0$, there exists $\delta =\delta(X, \theta, p, K, \Lambda, \epsilon)>0$  such that for any $\omega, \omega'\in \cK_\theta(p, K)$ with Ricci curvature bounded below by $-\Lambda$ and 
$$ 
\left\| \omega^n  -  (\omega')^n  \right\|_{L^1(X)} < \delta,
$$
we have 
\begin{equation}
\sup_{x,y\in X} \left|d_{\omega}(x, y) -  d_{\omega'}(x, y) \right| <  \epsilon.
\end{equation}
 In particular, 
\begin{equation}\label{ghclose}
\dgh((X, \omega), (X, \omega')) < \epsilon,
\end{equation}
 where $\dgh$ is the Gromov-Hausdorff distance between metric spaces. 

\end{corollary}


We encode a uniform lower Ricci bound by imposing a PSH condition on
the logarithmic volume ratio
$$f=-\log \frac{\omega^n}{\theta^n} .$$  
More precisely, we define 
\begin{equation}\label{pklspace}
\cK_\theta(p,K; \lambda)
\end{equation}
to be the set of $\omega\in \cK_\theta(p, K)$ satisfying 
\begin{equation}
-\log\frac{\omega^n}{\theta^n} \in \PSH(X, \lambda\theta). 
\end{equation}
We can rephrase the additional PSH-condition in terms of complex Monge-Amp\`ere equation by
$$\omega^n = e^{-f}\theta^n, ~f \in \PSH(X, \lambda\theta).$$
In fact, the Schwarz estimate shows that every
$\omega\in\cK_\theta(p,K;\lambda)$ satisfies $\omega\geq c\theta$ for some $c=c(X, \theta,p,K,\lambda)>0$ and hence
$\ric(\omega)\geq-\Lambda\omega$ for some $\Lambda=\Lambda(X, \theta, p, K, \lambda)$.

For each  K\"ahler metric $\omega\in \cK_\theta(p,K;\lambda)$, it induces a metric space $(X, \omega)$ and the Gromov-Hausdorff distance induces a topological structure on $\cK_\theta(p, K;\lambda)$. Since $\omega\in \cK_\theta(p,K;\lambda)$ has uniform lower Ricci bound and diameter bound by the work in \cite{GPSS1, GPSS2, GPSS3}, $\cK_\theta(p, K;\lambda)$ can be compactified as a compact space 
$$\overline{\cK_\theta(p,K; \lambda)}^{\dgh}$$ by the Gromov-Hausdorff topology.  On the other hand, if we let 
\begin{equation}\label{pshpk}
\PSH_{\theta, p,K}(X, \lambda \theta) =\{ f\in \PSH(X, \lambda\theta):~ \| e^{-f} \|_{L^p(X, \theta^n)}  \leq K, ~\int_X e^f\theta^n =\int_X\theta^n\}
\end{equation}
each  $f \in \PSH_{\theta, p,K}(X, \lambda \theta)$ induces, by Yau's solution to the Calabi conjecture \cite{Y}, a unique smooth
K\"ahler metric in $[\theta]$.  For non-smooth $f$, the corresponding  
solution is instead a positive current with bounded potential.  Thus, the comparison with $\PSH_{p,K}(X,\lambda\theta)/\mathbb R$ is always
understood on this normalized slice.  The $L^1$ topology is the
natural topology on the normalized quasi-plurisubharmonic data.
Corollary \ref{cor:main1} leads to the equivalence  between the Gromov-Hausdorff topology in $\overline{\cK_\theta(p, K)}^{\dgh}$ and  the $L^1$-topology of $\overline{\PSH_{\theta, p,K}(X, \lambda \theta)}^{L^1}$ modulo the automorphism group.  

Before we apply Theorem \ref{thm:main1} and Corollary \ref{cor:main1} to metric structures for singular  K\"ahler currents, let us discuss their potential extensions in the case when $X$ is a singular  K\"ahler space. In the case a $\mathbb{Q}$-Gorestein normal  K\"ahler space $X$ with klt singularities, we have to replace the volume form $\theta^n$ by an adapted volume measure $\Omega$. We will then replace $\cK_\theta(p,K;\lambda)$ by $\cK_{\theta, \Omega}(p,K;\lambda)$, the set of positive currents $\omega\in [\theta]$ with bounded local potentials satisfying
$$\| \frac{\omega^n}{\Omega}\|_{L^p(X, \Omega}\leq K, ~ -\log\frac{\omega^n}{\Omega} \in {\rm PSH}(X, \lambda\theta).$$

\begin{conjecture} Let $(X, \theta)$ be a compact normal  K\"ahler space with klt singularities equipped with a smooth  K\"ahler metric and an adapted volume measure $\Omega$. Then 
\begin{enumerate}
    \item Each  K\"ahler current $\omega\in \cK_\theta(p,K;\lambda)$ induces an RCD structure on $X$,
\medskip
    \item for any $\epsilon>0$, there exists $\delta>0$ such that for any $\omega, \omega'\in \cK_{\theta, \Omega}(p,K)$ with $$\| \omega^n - (\omega')^n|_{L^1(X)}< \delta,$$
we have 
$$\sup_{x,y\in X} |d_{\omega'}(x, y) - d_{\omega}(x,y)| < \epsilon.$$

\end{enumerate}   
\end{conjecture}
The conjecture holds in complex dimension in complex dimension 3 by combining the proof of Theorem \ref{thm:main1} with the RCD structure theory for 3-dimensional klt  K\"ahler spaces in \cite{FGS2} and the Holder continuity for  K\"ahler potentials on  K\"ahler RCD spaces in \cite{GKSS}.
\medskip

\subsection{Metric structures induced by  K\"ahler currents.~}

Let $\omega=\theta+\ddbar\varphi$ be a closed positive
$(1,1)$-current with
$\varphi\in\PSH(X,\theta)\cap L^\infty(X)$. If its Monge-Amp\`ere
volume measure satisfies
\begin{equation}\label{dbpsh}
 \left\|\frac{\omega^n}{\theta^n} \right\|_{L^p(X, \theta^n)}\leq K, ~ -\log \frac{\omega^n}{\theta^n} \in \PSH(X, \lambda \theta), 
\end{equation}
for some $p>1$,  the Ricci current of $\omega$ satisfies
$$\ric(\omega)\geq \ric(\theta) - \lambda \theta. $$
The approximation and Schwarz lemma in the proof of Theorem
\ref{thm:main2} also gives $\omega\geq c\theta$ as currents. Thus
$\omega$ is a K\"ahler current and
$$\ric(\omega) \geq - \Lambda\omega$$
for some $\Lambda=\Lambda(X,\theta, p, K,\lambda)>0$.

We define 
\begin{equation}\label{barkp} \overline{\cK_\theta(p,K;\lambda})
\end{equation}
to be the set of currents $\omega=\theta+\ddbar\varphi$ with
$\varphi\in\PSH(X,\theta)\cap L^\infty(X)$ satisfying
(\ref{dbpsh}). We will see that
$\overline{\cK_\theta(p,K;\lambda)}$ is a natural Ricci-controlled closure of
$\cK_\theta(p,K;\lambda)$.

If $\omega$ is a smooth  K\"ahler metric, then obviously $(X, \omega)$ induces a unique smooth metric space. What happens if $\omega$ is merely a current, which could be non-smooth everywhere?  

\begin{question} \label{question1} Does every $\omega \in \overline{\cK_\theta(p,K;\lambda)}$ induce a unique metric structure on $X$? 
    
\end{question}

The natural approach is the following regularization for the non-smooth $\omega$. By
B\l ocki-Ko\l odziej \cite{BK}, one can choose smooth functions
\[
 f_j\in\PSH(X,\lambda\theta),\qquad
 f_j\searrow-\log\frac{\omega^n}{\theta^n},
\]
and then solve for $\omega_j\in[\omega]$ satisfying
$$(\omega_j)^n = e^{-f_j+\epsilon_j} \theta^n$$
or equivalently
$$\ric(\omega_j) = \ric(\theta) + \ddbar f_j\geq - \Lambda\omega_j, $$
for some uniform $\Lambda>0$, where $\epsilon_j$ is the normalization
constant and $\epsilon_j\to0$.
The $\omega_j$ are smooth K\"ahler metrics with a uniform diameter
bound by \cite{GPSS1, GPSS2, GPSS3} and a uniform Ricci lower bound, so subsequential measured
Gromov-Hausdorff limits exist. Theorem \ref{thm:main1} will show that
the distance functions are uniformly Cauchy which leads to identifying the topology of the limit with
that of $X$.

\begin{theorem} \label{thm:main2} Let $(X, \theta)$ be a compact
K\"ahler manifold of dimension $n$. For
every $\omega\in \overline{\cK_\theta(p,K;\lambda)}$, there is a canonical
distance $d_\omega$ on $X$ such that
\begin{equation}\label{main2-rcd}
 (X,d_\omega, \omega^n)
 \quad\text{is a compact }\RCD(-\Lambda,2n)\text{ space}
\end{equation}
for some $\Lambda=\Lambda(X,\theta,p,K,\lambda)\geq 0$. Furthermore, the identity map
from $(X,d_\theta)$ to $(X,d_\omega)$ is a Lipschitz homeomorphism whose inverse is Holder continuous. 
    
\end{theorem}

The distance $d_\omega$ is canonical in the sense that it is independent of every admissible smooth regularization having common
$L^p$ and Ricci lower bounds. The uniqueness of Theorem \ref{thm:main2} immediately extends Theorem \ref{thm:main1}.

\begin{corollary} \label{cor:main2} Let $(X, \theta)$ be a compact
K\"ahler manifold of complex dimension $n$. For any
$\omega\in \overline{\cK_\theta(p,K;\lambda)}$ and any $\epsilon>0$, there
exists
$\delta=\delta(X,\theta,n,p,K,\lambda,\epsilon)>0$ such that for any
$\omega'\in \cK_\theta(p, K)$ satisfying
$$ 
\left\| (\omega')^n  -  \omega^n  \right\|_{L^1(X)} < \delta,
$$
we have
$$     %
d_{\omega'}(x, y) < d_{\omega}(x, y) + \epsilon
$$
for all $x,y\in X$.

\end{corollary}

Corollary \ref{cor:main1} can be similarly extended for $\omega, \omega'\in \overline{\cK_\theta(p,K;\lambda)}.$
 
We will now apply Theorem \ref{thm:main2} to a special family of  K\"ahler manifolds. Suppose $-K_X$ is pseudoeffective and klt in the sense of Definition
\ref{psklt}. Then, for any K\"ahler class $\alpha$ on $X$, there exists a
K\"ahler current $\omega\in\alpha$ such that $(X,\omega)$ induces a
unique $\RCD(0,2n)$ space homeomorphic to $X$.

\begin{theorem}\label{thm:main2'}
Let $(X,\theta)$ be a compact $n$-dimensional K\"ahler manifold.
If $-K_X$ is pseudoeffective and klt, there exist $p>1$,
$K>0$, $\lambda>0$, and a K\"ahler current
$\omega\in\overline{\cK(p,K;\lambda)}$ with non-negative Ricci
current such that $(X,\omega)$ induces a unique $\RCD(0,2n)$ space
homeomorphic to $X$.
\end{theorem}

This permits one to
study the topology of $X$ using tools for spaces with non-negative
Ricci curvature, as in \cite{FGSW1, FGSW2}, even when $X$ carries no
smooth K\"ahler metric with non-negative Ricci curvature. Although
$X$ is smooth, the resulting RCD metric may have a dense singular
set. 

Theorem \ref{thm:main2} and Theorem \ref{thm:main2'} can also be modified to the following statement. Suppose $$\ric(\omega) = - \ddbar\log \omega^n \geq - \Lambda \omega$$ for some $\omega\in \overline{\cK_\theta(p,K;\lambda)}$. Then it induces a unique distance $d_\omega$ on $X$ such that $(X, d_\omega, \omega^n)$ is a compact ${\rm RCD} (-\Lambda, 2n)$-space.

Corollary \ref{thm:main1} and Theorem \ref{thm:main2} also give a new regularization with control of distance for  K\"ahler currents in $\overline{\cK_\theta(p, K;\lambda)}$. Suppose $\varphi \in \PSH(X, \theta)$ with $f= - \log  \frac{(\theta+\ddbar \varphi)^n}{\theta^n} \in \PSH_{\theta, p,K}(X, \lambda \theta)$ for some $p>1$ and $K, \lambda>0$. Then there exists a sequence of smooth $\varphi_j \in \PSH(X, \theta)$ satisfying 
%
%
%
\begin{equation}
  \lim_{j\rightarrow \infty} \|\varphi_j -\varphi\|_{L^\infty(X)}=0,
  \qquad \lim_{j\rightarrow \infty} d_{GH}((X, \omega_j), (X, \omega)) = 0, 
\end{equation}
where $\omega_j=\theta+ \ddbar \varphi_j$.


\medskip

\subsection{Almost non-inflation under the K\"ahler-Ricci flow.} ~
Let $(X,\theta)$ be an $n$-dimensional compact K\"ahler manifold. We
consider the K\"ahler-Ricci flow
\begin{equation}\label{krflow}
\left\{
\begin{array}{l}
{ \displaystyle \ddt{\omega(t)} = -\ric(\omega(t)) ,}\\
\\
\omega(0)=\omega_0, 
\end{array} \right. 
\end{equation}
The weak-flow theory of \cite{ST3,DNL} applies when $\omega_0$ has a
bounded potential and an $L^p$ Monge-Amp\`ere density, $p>1$. It
gives a unique solution $\omega(t)$ on $X\times[0,T)$ satisfying
\begin{enumerate}
    \item
    $T=\sup\{t>0:[\omega_0]-tc_1(X)\text{ is a K\"ahler class}\}$,
    
    \medskip
    
    \item $\omega(t)$ is a smooth family of K\"ahler metrics on
    $X\times(0,T)$,
    
    \medskip
    
    \item the K\"ahler potentials of $\omega(t)$ converge uniformly
    to the potential of $\omega_0$ as $t\to0^+$.
    
    \end{enumerate}

Theorem \ref{thm:main3} controls the possible inflation of the
distance functions in this smoothing process.

\begin{theorem} \label{thm:main3}
Let $(X,\theta)$ be a compact K\"ahler manifold, and let $\omega(t)$
be the unique weak solution of the K\"ahler-Ricci flow
\eqref{krflow} starting from
$\omega_0\in\overline{\cK_\theta(p,K;\lambda)}$, where $p>1$ and
$K,\lambda>0$. Then
\begin{equation}
    \limsup_{t\rightarrow 0^+} \sup_{x,y\in X} \left( d_{\omega(t)}(x,y) - d_{\omega_0}(x,y) \right) \leq 0, 
\end{equation}

\end{theorem}

By Theorem \ref{thm:main2}, the initial current $\omega_0$ induces a
compact RCD metric structure on $X$.  Thus Theorem \ref{thm:main3}
is a distance non-inflation statement for the K\"ahler-Ricci
smoothing of this rough metric.

The definition of $\overline{\cK_\theta(p,K;\lambda)}$ can be extended to any normal  K\"ahler space with klt singularities with slight modifications. Let $X$ be an $n$-dimensional compact normal  K\"ahler space with klt singulariteis equipped with a smooth  K\"ahler metric $\theta$. One has to replace the reference volume measure $\theta^n$ by an adaptive volume measure $\Omega$ because of the severe degeneration of $\theta^n$ near the singularities of $X$.

\begin{conjecture} Let $X$ be a compact normal  K\"ahler variety with klt singularities. Suppose $\omega_0\in \overline{\cK_\theta(p, K;\lambda)}$ induces an RCD structure on $(X, \omega_0)$. Then the unique solution  $(X, \omega(t))$ induces a metric measure space homeomorphic to $X$ for each $t\in (0, T)$ with  
\begin{equation}\lim_{t\rightarrow 0^+}\dgh((X, \omega(t)), (X, \omega_0)) =0. 
\end{equation}

\end{conjecture}

We expect the conjecture to hold complex dimension 3 with canonical singularities. We further conjecture that if, in addition, the Ricci current of $\omega_0$ is bounded above, then the scalar curvature of $\omega(t)$ is uniformly bounded on $X\times [0, T-\epsilon]$ for any $\epsilon\in (0, T)$.

\bigskip
\noindent\textbf{Acknowledgments.} Corollary~\ref{cor:main1} was established as the main theorem by the authors in an earlier version of this paper. In a subsequent attempt to remove the Ricci curvature assumption and generalize it to Theorem~\ref{thm:main1}, ChatGPT pointed out an error in the argument and suggested inequality~\eqref{pwdevi} in Lemma \ref{ailemma}, which is one of the key ingredients for extending Corollary \ref{cor:main1} to Theorem~\ref{thm:main1}.


\section{Geometric estimates for complex Monge-Amp\`ere equations}

We review the basic estimates for complex Monge-Amp\`ere equations
used below. Throughout this section, $(X,\theta)$ is a compact
$n$-dimensional K\"ahler manifold. For simplicity, we assume 
$$[\theta]^n = \int_X \theta^n=1$$
throughout the rest of the paper. The set $\cK_\theta(p,K)$ consists of
the smooth K\"ahler metrics $\omega\in[\theta]$ satisfying the
$L^p$ bound \eqref{lp}, for some $p>1$.

\begin{theorem}   Suppose $\omega=\theta+\ddbar \varphi \in \cK_\theta(p, K)$ for $\varphi\in \PSH(X,\theta)$ with $\sup_X \varphi=0$. Then there exists $\alpha=\alpha(n, p)>0$ and $C=C(n,p,K,\alpha)>0$ such that
\begin{equation}\label{ahold}
    \|\varphi\|_{C^{\alpha}(X)} \leq C. 
\end{equation}
    
\end{theorem}

A result of \cite[Theorem~4.1 and Corollary~4.2]{Li} connects the
H\"older continuity of $\varphi$ to that of the distance function.
More precisely, with $\beta=\alpha/2$, there is a constant
$C=C(X,\theta,n,p,K)>0$ such that, uniformly for
$\omega\in\cK_\theta(p,K)$,
\begin{equation}\label{lhold}
d_\omega(x,y) \leq C d_\theta(x,y)^\beta
\qquad (x,y\in X).
\end{equation}
This distance Holder estimate immediately gives a uniform diameter bound for $\omega\in \cK_\theta(p,K)$.
It has been an open problem whether estimates (\ref{ahold}) and (\ref{lhold}) hold if the class $[\theta]$ is big and semi-positive instead of  K\"ahler, which inevitably appears as one studies geometric and analytic structures of singularities.  The loss of positivity poses essential obstacles as the approaches in the  K\"ahler generally fail. Recent progress has been made in \cite{GPSS1, GPSS2, GPSS3, GKSS, GSS} to build a geometric theory of complex Monge-Amp\`ere equations with and without curvature. For example, the uniform diameter is established \cite{GPSS1} without any curvature assumptions, while a Holder estimate is proved \cite{GKSS} with curvature assumptions.

We now return to the classical stability theorem for complex Monge-Amp\`ere equation established in \cite{Kol2}. Our goal is to turn the $L^\infty$-estimate for potentials to the $L^\infty$-estimate for distance functions.
\begin{theorem} For any $\epsilon>0$, there exists $\delta=\delta(n, p, K, \epsilon)>0$ such that for any $\omega=\theta+\ddbar\varphi$ and  $\omega' =\theta+\ddbar \varphi'\in \cK_\theta(p,K)$ satisfying
$$\|(\omega')^n -\omega^n \|_{L^1(X)}\leq \delta,$$
and normalize $\sup_X\varphi=\sup_X\varphi'=0$, then
$$\|\varphi'-\varphi\|_{L^\infty(X)}< \epsilon.$$

\end{theorem}


\section{An elementary inequality}

In this section, we formulate an elementary inequality that is related to absolute mean deviation. Suppose 
$$\lambda_1,\dots,\lambda_n>0.$$ 
We define 
$$S=\sum_{j=1}^n\lambda_j,\qquad
P=\prod_{j=1}^n\lambda_j$$
and we let 
\[
(S-n)_+=\max\{S-n,0\},\qquad
(S-n)_-=\max\{n-S,0\}=(n-S)_+.
\]

\begin{lemma}\label{ailemma} There is a constant $C_n$, depending only on $n$, such that
\begin{equation}\label{pwdevi}
\sum_{j=1}^n |\lambda_j-1|
\le
C_n\left(
(S-n)_+
+
\sqrt{(S-n)_+ + |P-1|}
\right).
\end{equation}

\end{lemma}

\begin{proof}
    Let
\[
\delta=(S-n)_+ + |P-1|,
\]
and we will prove
\[
E\le C_n\left((S-n)_+ + \sqrt{\delta}\right).
\]
We break the proof into two cases. 
\begin{enumerate}

\item We consider the case of large trace when 
$$S \geq n+1. $$
We immediately have 
\[
E=\sum_{\lambda_j\ge1}(\lambda_j-1)+\sum_{\lambda_j<1}(1-\lambda_j)
\le (S-n)+2n
\]
from the assumptions. 
Hence
\[
E\le (2n+1)(S-n),
\]
 the desired estimate follows with $C_n = 2n+1$.

\medskip

\item We now consider the case when   $$S<n+1.$$
This immediately implies $\lambda_j<n+1$ every where on $X$ for each $j=1, ..., n$. Hence 
$E\leq n^2$ is uniformly bounded. 

\begin{claim} $\delta=0$ if and only $$\lambda_1=...=\lambda_n =1. $$
    
\end{claim}

\begin{proof} if $\delta=0$, then
$S\le n$ and $P=1$. By the AM-GM inequality, we have
\[
P^{1/n}\le \frac Sn\le1,
\]
and equality implies $\lambda_1=\cdots=\lambda_n=1$.

\end{proof}

Fix a small neighborhood $U$ of $(1,\ldots,1)$. On the compact set
\[
 \{\lambda_j\geq0,\ S\leq n+1\}\setminus U,
\]
the continuous function $\delta$ has no zero and hence has a positive
minimum, while $E$ is bounded. Therefore there exists $C=C(n,U)>0$
such that
\[
E\le C_n\sqrt{\delta}. 
\]
Therefore it remains to prove the desired inequality in a fixed small neighborhood $U$ of $(1,\ldots,1)$.

We let 
$$\lambda_j=1+x_j$$  and define
\[
T=\sum_{j=1}^n x_j=S-n,\qquad
Q=\sum_{j=1}^n x_j^2.
\]

\begin{claim} For a fixed sufficiently small neighborhood $U$ of $(1,..., 1)$, there exists $c=c(n, U)>0$ such that 
$$ \delta \geq c Q. $$
    
\end{claim}

\begin{proof} Using the expansion of 
\[
\log(1+x)=x-\frac{x^2}{2}+O(|x|^3),
\]
we obtain
\[
\log P = \sum_{j=1}^n \log (1+x_j)
=
T-\frac12Q
+O\!\left( \sum_{j=1}^n|x_j|^3\right).
\]
After shrinking the neighborhood $U$, we have 
\[
\log P\le T-\frac14Q.
\]

If $T_+\ge \frac{1}{8}Q$, then
\[
\delta\ge T_+\ge \frac18 Q.
\]

Otherwise $T_+<\frac{1}{8}Q$, so
\[
\log P\le -\frac18 Q.
\]
Hence $P<1$. After shrinking $U$ once more, $Q$ is small and
\[
 1-P\geq 1-e^{-Q/8}\geq \frac{1}{16}Q.
\]
Consequently,
\[
 \delta\geq |P-1|\geq \frac{1}{16}Q.
\]
 
Thus in either case
\[
\delta\ge c Q
\]
and we have proved the claim.

\end{proof}
\noindent Finally, we can complete the proof in the second case because
\[
E=\sum_{j=1}^n |x_j|
\le
\sqrt n\,Q^{1/2}
\le
C\sqrt{\delta}
\]

\end{enumerate}

Combining the two cases, we have proved the lemma.

\end{proof}


\section{An integral estimate for the metric deviation}

Let $(Z,d\mu)$ be a probability space and let
$\lambda_1,\ldots,\lambda_n:Z\to\mathbb R_{>0}$ be measurable. Put
$$E(x)=\sum_{j=1}^n|\lambda_j(x)-1|,\qquad
S(x)=\sum_{j=1}^n\lambda_j(x),\qquad
P(x)=\prod_{j=1}^n\lambda_j(x).$$

\begin{lemma} Suppose 
$$\int_Z(S-n)\,d\mu=0,$$
then 
\begin{equation}\label{intdevi}
\int_Z E\,d\mu \leq C_n \left( \int_Z |P-1|\,d\mu
+ \left(\int_Z |P-1|\,d\mu \right)^{1/2} \right).
\end{equation}

\end{lemma}

\begin{proof} Integrating the pointwise estimate (\ref{pwdevi}), we have  
\[
\int_Z E\,d\mu
\le
C_n\left[
\int_Z (S-n)_+\,d\mu
+
\int_Z \sqrt{(S-n)_+ + |P-1|}\,d\mu
\right].
\]
By Cauchy-Schwarz, we have 
\[
\int_Z E\,d\mu
\le
C_n
\left[
\int_Z (S-n)_+\,d\mu
+
\left(
\int_Z (S-n)_+\,d\mu
+
\int_Z |P-1|\,d\mu
\right)^{1/2}
\right].
\]
Since
\[
\int_Z (S-n)\,d\mu=0, 
\]
we have 
\[
\int_Z (S-n)_+\,d\mu
=
\int_Z (S-n)_-\,d\mu.
\]

On the region $U=\{S\leq n\}\subset Z$, the AM-GM inequality gives
\[
P^{1/n}\le \frac{S}{n} \leq 1,
\]
hence
\[
n-S
\le
n(1-P^{1/n})
\le
n(1-P)
\le
n|P-1|.
\]
This implies 
\[
(S-n)_-\le n|P-1|
\]
on $U$ and so 
\[
\int_Z (S-n)_+\,d\mu
=
\int_Z (S-n)_-\,d\mu = \int_U(S- n)_- \, d\mu
\le
n\int_Z |P-1|\,d\mu.
\]

Substituting into the previous estimate yields
\[
\int_Z E\,d\mu
\le
C_n
\left(
\int_Z |P-1|\,d\mu
+
\left(
\int_Z |P-1|\,d\mu
\right)^{1/2}
\right).
\]
\end{proof}


We will apply the estimate (\ref{intdevi}) to   K\"ahler metrics in the same cohomology class. 
Let $\omega$ and $\omega'$ be two smooth K\"ahler metrics on an $n$-dimensional compact  K\"ahler manifold $X$ satisfying
$$[\omega]=[\omega'].$$
Put $V=\int_X\omega^n=\int_X(\omega')^n$ and let
$\lambda_1,\dots,\lambda_n$ be the eigenvalues of $\omega'$ with respect to $\omega$. We will apply (\ref{intdevi}) with   
\[
S=\operatorname{tr}_{\omega}(\omega')=\sum_{j=1}^n \lambda_j,
\qquad
P=\frac{(\omega')^n}{\omega^n} =\prod_{j=1}^n \lambda_j,
\qquad
d\mu= \omega^n.
\]

The quantity of interest is the following deviation of $\omega'$ from $\omega$ 
\begin{equation} \label{mdevi}
E_{\omega,\omega'} =  
\sum_{j=1}^n |\lambda_j-1|
.
\end{equation}
Since $[\omega]=[\omega']$,
\[
\int_X
(\operatorname{tr}_{\omega}(\omega')-n)\,d\mu
= n\int_X(\omega'-\omega)\wedge\omega^{n-1}=0.
\]
By applying (\ref{intdevi}), we have proved following lemma. 

\begin{lemma}
\begin{equation}\label{lem:mdevi}
\int_X E_{\omega,\omega'}\,d\mu
\le
C_n
\left(
 \left\|(\omega')^n-\omega^n\right\|_{L^1(X)}
+
\left(
 \left\|(\omega')^n-\omega^n\right\|_{L^1(X)}
\right)^{1/2}
\right).
\end{equation}

\end{lemma}


\section{Proof of Theorem \ref{thm:main1}}

Fix $\omega,\omega'\in\cK_\theta(p,K)$ and put
\[
 \int_X\theta^n=1,\qquad
 \mu= \omega^n,\qquad
 q= \|(\omega')^n-\omega^n\|_{L^1(X)}.
\]
Let $E=E_{\omega,\omega'}$ be the deviation defined in
\eqref{mdevi}. For $x,y\in X$, let $\Gamma_\omega(x,y)$ be the
family of unit-speed minimizing $\omega$-geodesics from $x$ to $y$ and
define
\begin{equation}\label{segment-functional}
 {\mathcal F}_E(x,y)
 =\inf_{\gamma\in\Gamma_\omega(x,y)}
   \int_\gamma E\,ds_\omega.
\end{equation}

We use the following standard form of the Cheeger-Colding segment
inequality.

\begin{lemma}\label{lem:seg}
Let $(M^m,g)$ be compact with $\ric(g)\geq-\lambda g$. If
$A,B\subset B_g(a,R)$ and $h\geq0$ is measurable, then
\begin{equation}\label{segment-inequality}
 \int_{A\times B}
 \inf_{\gamma\in\Gamma_g(x,y)}\int_\gamma h\,ds_g\,
 dV_g(x)dV_g(y)
 \leq C\bigl(\Vol_g(A)+\Vol_g(B)\bigr)
 \int_{B_g(a,2R)}h\,dV_g,
\end{equation}
where $C=C(m,\lambda,R)$.
\end{lemma}

\begin{lemma}\label{lem:seg1}
If $\ric(\omega)\geq-\lambda\omega$, then
\begin{equation}\label{integrated-segment-deviation}
 \int_{X\times X}{\mathcal F}_E\,d(\mu\otimes\mu)
 \leq C\bigl(q+q^{1/2}\bigr),
\end{equation}
where $C$ depends only on the fixed background data, $p,K$, and
$\lambda$.
\end{lemma}

\begin{proof}
The distance H\"older estimate \eqref{lhold} gives a uniform diameter
bound $D$ for the metrics in $\cK_\theta(p,K)$. Apply
Lemma \ref{lem:seg} with $A=B=X$, $R>D$, and $h=E$. After dividing by
$V^2$, the right-hand side is bounded by
$C\int_XE\,d\mu$. Estimate \eqref{lem:mdevi} proves
\eqref{integrated-segment-deviation}.
\end{proof}

For $r>0$, define the one-sided bad set
\begin{equation}\label{bad-minus}
 \cB^-(r)=
 \{(x,y)\in X\times X:
 d_{\omega'}(x,y)>d_\omega(x,y)+r\}.
\end{equation}

\begin{lemma}\label{lem:vollow}
Under the assumptions of Lemma \ref{lem:seg1},
\begin{equation}\label{bad-set-volume}
 (\mu\otimes\mu)(\cB^-(r))
 \leq \frac{C}{r}\bigl(q+q^{1/2}\bigr).
\end{equation}
\end{lemma}

\begin{proof}
Let $\gamma$ be any unit-speed minimizing $\omega$-geodesic from
$x$ to $y$. Diagonalizing $\omega'$ with respect to $\omega$ gives
\[
 g_{\omega'}(\dot\gamma,\dot\gamma)\leq 1+E(\gamma).
\]
Since $\sqrt{1+t}\leq1+t$ for $t\geq0$,
\[
 d_{\omega'}(x,y)
 \leq L_{\omega'}(\gamma)
 \leq d_\omega(x,y)+\int_\gamma E\,ds_\omega.
\]
This holds for every $\gamma\in\Gamma_\omega(x,y)$, and hence
\[
 d_{\omega'}(x,y)-d_\omega(x,y)\leq{\mathcal F}_E(x,y).
\]
Thus $\cB^-(r)\subset\{{\mathcal F}_E>r\}$, and
\eqref{bad-set-volume} follows from Markov's inequality and
Lemma \ref{lem:seg1}.
\end{proof}

\begin{lemma}\label{lemformain1}
Let $\omega\in\cK_\theta(p,K)$ satisfy
$\ric(\omega)\geq-\lambda\omega$. For every $\eta>0$, there is
$\delta=\delta(X,\theta,n,p,K,\lambda,\eta)>0$ such that
\[
 \|(\omega')^n-\omega^n\|_{L^1(X)}<\delta,\qquad
 \omega'\in\cK_\theta(p,K),
\]
implies
\begin{equation}\label{one-sided-uniform}
 d_{\omega'}(x,y)\leq d_\omega(x,y)+\eta
 \qquad\text{for all }x,y\in X.
\end{equation}
\end{lemma}

\begin{proof}
Write $\omega=\theta+\ddbar\varphi$ with $\sup_X\varphi=0$.
Ko\l odziej's estimate bounds $\operatorname{osc}_X\varphi$. The Schwarz lemma  and the assumption
$\ric(\omega)\geq-\lambda\omega$ give, for a background constant
$C_0$,
\[
 \Delta_\omega\log\tr_\omega\theta
 \geq-C_0\tr_\omega\theta-\lambda.
\]
Since $\Delta_\omega\varphi=n-\tr_\omega\theta$, the maximum principle
applied to $\log\tr_\omega\theta-A\varphi$, with $A>C_0$, gives
\begin{equation}\label{fixed-schwarz}
 \tr_\omega\theta\leq C,\qquad\text{hence}\qquad
 \omega\geq c\theta.
\end{equation}
Combining \eqref{fixed-schwarz} with \eqref{lhold}, we obtain constants
$A_0>0$ and $\beta>0$, independent of $\omega'$, such that
\begin{equation}\label{cross-holder}
 d_{\omega'}(z,z')\leq A_0d_\omega(z,z')^\beta
 \qquad(z,z'\in X).
\end{equation}

Choose $\rho>0$ so small that
\begin{equation}\label{rho-choice}
 2\rho+2A_0\rho^\beta<\frac{\eta}{2}.
\end{equation}
Suppose, contrary to \eqref{one-sided-uniform}, that
\[
 d_{\omega'}(x,y)>d_\omega(x,y)+\eta
\]
for some $x,y\in X$. If
$x'\in B_\omega(x,\rho)$ and $y'\in B_\omega(y,\rho)$, the triangle
inequality, \eqref{cross-holder}, and \eqref{rho-choice} give
\[
\begin{split}
 d_{\omega'}(x',y')
 &>d_{\omega'}(x,y)-2A_0\rho^\beta\\
 &>d_\omega(x',y')+\eta-2\rho-2A_0\rho^\beta\\
 &>d_\omega(x',y')+\frac{\eta}{2}.
\end{split}
\]
Therefore
\begin{equation}\label{product-ball-bad}
 B_\omega(x,\rho)\times B_\omega(y,\rho)
 \subset\cB^-(\eta/2).
\end{equation}

The Ricci lower bound, the fixed total volume, and the uniform diameter
bound imply by Bishop-Gromov comparison that
\begin{equation}\label{ball-lower}
 \mu(B_\omega(z,\rho))\geq v_\rho>0
 \qquad(z\in X),
\end{equation}
where $v_\rho$ depends only on the fixed data, $\lambda$, and $\rho$.
Equations \eqref{product-ball-bad} and \eqref{ball-lower} yield
\[
 (\mu\otimes\mu)(\cB^-(\eta/2))\geq v_\rho^2.
\]
On the other hand, Lemma \ref{lem:vollow} makes the left-hand side
strictly smaller than $v_\rho^2$ once $q$, and hence $\delta$, is
sufficiently small. This contradiction proves the lemma.
\end{proof}

Applying Lemma \ref{lemformain1} with $\eta=\epsilon/2$ gives
\[
 d_{\omega'}(x,y)\leq d_\omega(x,y)+\frac{\epsilon}{2}
 <d_\omega(x,y)+\epsilon,
\]
which proves Theorem \ref{thm:main1}.

\begin{proof}[Proof of Corollary \ref{cor:main1}]
Apply Lemma \ref{lemformain1} with $\eta=\epsilon/2$ to
$(\omega,\omega')$ and then to $(\omega',\omega)$. The constants are
uniform under the common Ricci lower bound, and hence
\[
 \sup_{x,y\in X}|d_\omega(x,y)-d_{\omega'}(x,y)|
 \leq\frac{\epsilon}{2}<\epsilon.
\]
The Gromov-Hausdorff estimate follows by using the identity
correspondence on $X$.
\end{proof}


\section{Proof of Theorem \ref{thm:main2}}

We write
\[
 \omega=\theta+\ddbar\varphi,\qquad
 f=-\log\frac{\omega^n}{\theta^n},\qquad
 \sup_X\varphi=0.
\]
Thus
\begin{equation}\label{singular-ma-equation}
 (\theta+\ddbar\varphi)^n=e^{-f}\theta^n,
 \qquad f\in\PSH(X,\lambda\theta),\qquad
 \|e^{-f}\|_{L^p(X,\theta^n)}\leq K.
\end{equation}

By the smooth monotone regularization theorem of
B\l ocki-Ko\l odziej \cite[Theorem~1]{BK}, there are
\[
 f_j\in C^\infty(X)\cap\PSH(X,\lambda\theta),
 \qquad f_j\searrow f.
\]
(If $\lambda=0$, then $f$ is constant and one simply takes
$f_j=f$.)
Put
\begin{equation}\label{normalizing-epsilon}
 I_j=\int_Xe^{-f_j}\theta^n,\qquad
 \epsilon_j=\log\frac{1}{I_j},\qquad
 h_j=e^{\epsilon_j-f_j}.
\end{equation}
Because $e^{-f_j}\nearrow e^{-f}$ and
$\int_Xe^{-f}\theta^n=1$, monotone convergence gives
$I_j\nearrow 1$ and $\epsilon_j\searrow0$. Moreover,
\begin{equation}\label{normalized-density-convergence}
 \|h_j-e^{-f}\|_{L^1(X,\theta^n)}\longrightarrow0,\qquad
 \sup_j\|h_j\|_{L^p(X,\theta^n)}\leq K_0
\end{equation}
for a fixed $K_0$. Indeed, $e^{-f_j}\leq e^{-f}$ and
\[
\begin{split}
 \|h_j-e^{-f}\|_{L^1(X, \theta^n)}
 &\leq e^{\epsilon_j}\|e^{-f_j}-e^{-f}\|_{L^1(X, \theta^n)}
 +(e^{\epsilon_j}-1)\|e^{-f}\|_{L^1(X, \theta^n)},
\end{split}
\]
After discarding finitely many terms, one may take $K_0=2K$.

By Yau's theorem, there is a unique smooth normalized solution
\begin{equation}\label{maeqn2}
 \omega_j=\theta+\ddbar\varphi_j>0,\qquad
 \omega_j^n=h_j\theta^n,\qquad
 \sup_X\varphi_j=0.
\end{equation}
Ko\l odziej's uniform estimate and stability theorem, together with
\eqref{normalized-density-convergence}, give
\begin{equation}\label{potential-regularization}
 \sup_j\|\varphi_j\|_{L^\infty(X)}<\infty,\qquad
 \|\varphi_j-\varphi\|_{L^\infty(X)}\longrightarrow0.
\end{equation}
In particular, the smooth metrics $\omega_j$ genuinely regularize the
given current $\omega$.

\begin{lemma}\label{lemsch}
There are constants $c>0$ and $\Lambda>0$, independent of $j$, such
that
\begin{equation}\label{schw2}
 \omega_j\geq c\theta,
 \qquad
 \ric(\omega_j)\geq-\Lambda\omega_j.
\end{equation}
\end{lemma}

\begin{proof}
Since $f_j\in\PSH(X,\lambda\theta)$,
\begin{equation}\label{ricci-theta-lower}
 \ric(\omega_j)=\ric(\theta)+\ddbar f_j
 \geq\ric(\theta)-\lambda\theta\geq-C_0\theta.
\end{equation}
Set $S_j=\tr_{\omega_j}\theta$. The Schwarz lemma, using the
fixed upper bound for the holomorphic bisectional curvature of
$\theta$ and \eqref{ricci-theta-lower}, gives
\begin{equation}\label{chern-lu-main2}
 \Delta_{\omega_j}\log S_j\geq-C_1S_j
\end{equation}
with $C_1$ independent of $j$. Since
$\Delta_{\omega_j}\varphi_j=n-S_j$, for $A>C_1$ we have
\[
 \Delta_{\omega_j}
 \bigl(\log S_j-A\varphi_j\bigr)
 \geq(A-C_1)S_j-An.
\]
At a maximum point of $\log S_j-A\varphi_j$, this bounds $S_j$.
Comparing the maximum with an arbitrary point and using
\eqref{potential-regularization} then gives
$\sup_XS_j\leq C$. Hence $\omega_j\geq c\theta$. Combining this with
\eqref{ricci-theta-lower} gives
$\ric(\omega_j)\geq-(C_0/c)\omega_j$; take
$\Lambda=C_0/c$.
\end{proof}

Passing to the weak limit in $\omega_j\geq c\theta$ also gives
\begin{equation}\label{singular-current-lower}
 \omega\geq c\theta.
\end{equation}

\begin{lemma}\label{lem:canonical-distance}
The distances $d_{\omega_j}$ converge uniformly on $X\times X$ to a
distance $d_\omega$. The limit is independent of every admissible
smooth regularization.
\end{lemma}

\begin{proof}
Equation \eqref{normalized-density-convergence} says precisely that
\begin{equation}\label{tv-current-approximation}
 \|\omega_j^n-\omega^n\|_{L^1(X)}\longrightarrow0.
\end{equation}
All $\omega_j$ have the common $L^p$ bound $K_0$ and Ricci lower bound
in \eqref{schw2}. Given $\eta>0$, apply Theorem
\ref{thm:main1} first to $(\omega_j,\omega_k)$ and then with the two
metrics interchanged. The constant is uniform in $j,k$, and
\eqref{tv-current-approximation} therefore gives
\[
 \sup_{X\times X}|d_{\omega_j}-d_{\omega_k}|\longrightarrow0.
\]
Let $d_\omega$ be the uniform limit. Symmetry and the triangle
inequality pass to the limit, while \eqref{schw2} gives
\begin{equation}\label{limit-distance-lower}
 d_\omega\geq\sqrt c\,d_\theta.
\end{equation}
Thus $d_\omega$ separates points.

Here an admissible smooth regularization means a sequence of smooth
K\"ahler metrics $\widehat\omega_k\in[\theta]$ whose volume densities
have a common $L^p$ bound $\widehat K$, whose Ricci curvatures have a common lower
bound
$\ric(\widehat\omega_k)\geq-\widehat\Lambda\widehat\omega_k$, and for
which
\[
 \|\widehat\omega_k^n-\omega^n\|_{L^1(X)}\longrightarrow0.
\]
For such a sequence,
\[
 \|\widehat\omega_k^n-\omega_j^n\|_{L^1(X)}
 \leq\|\widehat\omega_k^n-\omega^n\|_{L^1(X)}
 +\|\omega_j^n-\omega^n\|_{L^1(X)}\longrightarrow0.
\]
Set
$K_*:=\max\{K_0,\widehat K\}$ and
$\Lambda_*:=\max\{\Lambda,\widehat\Lambda\}$.  Theorem
\ref{thm:main1}, applied in both directions with these common
constants, shows that for every $\eta>0$ there are $j_0,k_0$ such
that
\[
 \sup_{X\times X}
 |d_{\widehat\omega_k}-d_{\omega_j}|<\eta
 \qquad(j\geq j_0,\ k\geq k_0).
\]
Letting first $j\to\infty$ shows that
$d_{\widehat\omega_k}\to d_\omega$ uniformly. This proves the asserted
independence. If $\omega$ is smooth, the constant sequence
$\widehat\omega_k=\omega$ shows that $d_\omega$ is its ordinary
Riemannian distance.
\end{proof}

\begin{lemma}\label{thm2pf}
With $\mu= \omega^n$,
\begin{equation}\label{main2-mgh}
 (X,d_{\omega_j}, \omega_j^n)
 \longrightarrow (X,d_\omega,\mu)
\end{equation}
in measured Gromov-Hausdorff topology. The limit is a compact
$\RCD(-\Lambda,2n)$ space homeomorphic to $X$.
\end{lemma}

\begin{proof}
First, \eqref{schw2} and the uniform distance H\"older estimate
\eqref{lhold} pass to the limit:
\begin{equation}\label{limit-bi-holder}
 \sqrt c\,d_\theta(x,y)
 \leq d_\omega(x,y)
 \leq C d_\theta(x,y)^\beta.
\end{equation}
The identity maps in both directions are therefore continuous, so
$(X,d_\omega)$ is homeomorphic to $(X,d_\theta)$ and, in particular,
is compact.  Moreover, writing $\omega=c\theta+T$ with $T\geq0$, the
Bedford-Taylor products of these currents with bounded potentials
satisfy
\[
 \omega^n=(c\theta+T)^n\geq c^n\theta^n.
\]
Thus $\mu= \omega^n$ is a Borel probability measure with full
support for the $d_\omega$ topology.

The identity correspondences now have distortion tending to zero by
Lemma \ref{lem:canonical-distance}, and the normalized measures
converge in total variation by \eqref{tv-current-approximation}. This
proves \eqref{main2-mgh}. Each smooth metric-measure space is
$\RCD(-\Lambda,2n)$ by \eqref{schw2}; stability of the RCD condition
\cite{GMS} gives the same bound for $(X,d_\omega,\mu)$.
\end{proof}

Lemmas \ref{lem:canonical-distance} and \ref{thm2pf} prove Theorem
\ref{thm:main2}.

\begin{proof}[Proof of Corollary \ref{cor:main2}]
Let $K_0$ and $\Lambda$ be the uniform constants in the preceding
construction, enlarged so that $K_0\geq K$. Apply Theorem
\ref{thm:main1}, with reference metrics satisfying
$\ric\geq-\Lambda\omega$, to obtain
$\delta_0=\delta_0(X,\theta,n,p,K_0,\Lambda,\epsilon/3)>0$.
Choose $j$ sufficiently large that
\[
 \sup_{X\times X}|d_{\omega_j}-d_\omega|<\frac{\epsilon}{3},
 \qquad
 \|\omega_j^n-\omega^n\|_{L^1(X)}<\frac{\delta_0}{2}.
\]
If
$\|(\omega')^n-\omega^n\|_{L^1(X)}<\delta_0/2$, then
$\|(\omega')^n-\omega_j^n\|_{L^1(X)}<\delta_0$. Theorem
\ref{thm:main1}, with $\omega_j$ as the reference metric, gives
\[
 d_{\omega'}(x,y)
 <d_{\omega_j}(x,y)+\frac{\epsilon}{3}
 <d_\omega(x,y)+\epsilon.
\]
Thus the assertion holds with $\delta=\delta_0/2$.
\end{proof}

\section{Proof of Theorem \ref{thm:main2'}}

We will prove Theorem \ref{thm:main2'} in this section. First, we recall the following construction of cut-off functions give vanishing $W^{1,2}$-capacity for analytic subvarieties in $X$ \cite{So14, FGS2}.  
\begin{lemma}\label{lem:analytic-set-cutoffs}
Let $Z\subset X$ be a proper analytic subset and let
$\gamma=\theta+\ddbar\psi$ be a positive current with bounded
potential.  Suppose that $\gamma\geq c\theta$ on $X$ and that
$\gamma$ is a smooth K\"ahler metric on $U=X\setminus Z$.  There
are functions $\chi_k\in C_c^\infty(U)$ such that
\begin{equation}
 0\leq\chi_k\leq1,\qquad
 \chi_k\equiv1\quad\hbox{on every }V \Subset U
 \hbox{ for all sufficiently large }k,
\end{equation}
and
\begin{equation}\label{general-analytic-set-cutoff-energy}
 \int_U\left(
 |\nabla\chi_k|_\gamma^2+|\Delta_\gamma\chi_k|
 \right)\gamma^n\longrightarrow0.
\end{equation}
\end{lemma}

\begin{theorem} 
\label{thm:current-bochner}
Let $(X,\theta)$ be a compact K\"ahler manifold of complex dimension
$n$, and let
\begin{equation}
 \omega=\theta+\ddbar\varphi\in
 \overline{\cK_\theta(p,K;\lambda)}.
\end{equation}
Write
\begin{equation}
 f=-\log\frac{\omega^n}{\theta^n},\qquad
 \int_X\omega^n=1,
 \qquad \mu= \omega^n,
\end{equation}
and let $(X,d_\omega,\mu)$ be the canonical metric measure space
given by Theorem \ref{thm:main2}.  If
\begin{equation}\label{ric-current-nonnegative}
 \ric(\omega)=\ric(\theta)+\ddbar f\geq0
\end{equation}
as a current, then $(X,d_\omega,\mu)$ is an
$\RCD(0,2n)$ space.

\end{theorem}

\begin{proof}
We retain analytic singularities during the regularization.  This is
essential because the class $c_1(X)$ is only assumed to contain the
positive current in \eqref{ric-current-nonnegative}, and need not be
nef.

By the Bergman-kernel form of Demailly's regularization theorem
\cite[Proposition 3.7]{Dem92}, there are numbers
$\varepsilon_j\searrow0$ and
quasi-plurisubharmonic functions $f_j$ with analytic singularities
such that
\begin{equation}\label{demailly-current-approximation}
 \ric(\theta)+\ddbar f_j\geq-\varepsilon_j\theta,
 \qquad f_j\longrightarrow f\quad\hbox{almost everywhere and in }L^1,
\end{equation}
and
\begin{equation}\label{demailly-lower-control}
 f_j\geq f.
\end{equation}
In particular,
\begin{equation}
 e^{-p f_j}\leq e^{-p f}.
\end{equation}
Consequently the functions $e^{-f_j}$ are uniformly bounded in
$L^p(X,\theta^n)$ and converge to $e^{-f}$ in $L^1(X,\theta^n)$.

Let $b_j$ be determined by
\begin{equation}
 e^{b_j}\int_Xe^{-f_j}\theta^n=1.
\end{equation}
Then $b_j\to0$.  Let $\varphi_j\in\PSH(X,\theta)\cap L^\infty(X)$,
normalized by $\sup_X\varphi_j=0$, solve
\begin{equation}\label{analytic-singularity-ma}
 \omega_j^n=(\theta+\ddbar\varphi_j)^n
 =e^{-f_j+b_j}\theta^n.
\end{equation}
Let $Z_j$ be the analytic singular set of $f_j$ and put
$U_j=X\setminus Z_j$.  We justify the estimates and the regularity of
$\omega_j$ using a specific inner regularization.  Choose smooth,
convex, nondecreasing functions $M_k\colon\mathbb R\to\mathbb R$ such
that
\[
 0\leq M_k'\leq1,\qquad M_k''\geq0,
 \qquad M_k(t)\geq t,
\]
with $M_k$ constant on $(-\infty,-k-1]$ and $M_k(t)=t$ for
$t\geq-k+1$.  Since $f_j$ has analytic singularities,
\[
 f_{j,k}:=M_k(f_j)
\]
extends smoothly across $Z_j$ and equals $f_j$ on every compact
subset of $U_j$ for all sufficiently large $k$.  A fixed constant
$B$, independent of $j,k$, exists by
\eqref{demailly-current-approximation} and the fixed upper bound for
$\ric(\theta)$ with respect to $\theta$; moreover,
\[
 \ddbar f_{j,k}
 =M_k'(f_j)\ddbar f_j
  +M_k''(f_j)\sqrt{-1}\partial f_j\wedge\dbar f_j
 \geq-B\theta.
\]
Moreover, $f_{j,k}\geq f_j$, so the $L^p$ bounds are uniform.

We  solve the following smooth equations
$$\omega_{j,k}^n=e^{-f_{j,k}+b_{j,k}}\theta^n, ~\sup_X\varphi_{j,k}=0$$ ~
with the normalizations given by
$$ e^{b_{j,k}}\int_Xe^{-f_{j,k}}\theta^n=1 ,$$
where $\omega_{j,k}=\theta+\ddbar\varphi_{j,k}>0$.
We have 
$b_{j,k}\to b_j$ by dominated convergence, and after passing to a tail, the normalized
right hand sides still satisfy a common $L^p$ bound.   Ko\l odziej's estimates
and Schwarz lemma used in Theorem \ref{thm:main2} therefore give
constants $C>0$ and $c>0$, independent of $j,k$, such that
\[
 \|\varphi_{j,k}\|_{L^\infty(X)}\leq C,
 \qquad \omega_{j,k}\geq c\theta.
\]
Letting $k\to\infty$ and using Ko\l odziej stability gives
\begin{equation}\label{analytic-approximation-estimates}
 \|\varphi_j\|_{L^\infty(X)}\leq C,
 \qquad \omega_j\geq c\theta.
\end{equation}
The stability theorem for the complex Monge-Amp\`ere equation gives
\begin{equation}\label{analytic-approximation-potential-convergence}
 \|\varphi_j-\varphi\|_{L^\infty(X)}\longrightarrow0.
\end{equation}

If $V\Subset V'\Subset U_j$, the functions $f_{j,k}$ eventually equal
$f_j$ on $V'$ and $b_{j,k}\to b_j$, so the right hand sides converge
there in $C^\infty$.  The bound $\omega_{j,k}\geq c\theta$ and the
determinant equation give uniform ellipticity on $V'$.  Local
Evans-Krylov and Schauder estimates therefore give smooth convergence
of $\omega_{j,k}$ to $\omega_j$ on $V$.  In particular, $\omega_j$ is
smooth on $U_j$.
Equations \eqref{demailly-current-approximation} and
\eqref{analytic-approximation-estimates} imply there that
\begin{equation}\label{regular-locus-ricci}
 \ric(\omega_j)=\ric(\theta)+\ddbar f_j
 \geq-\varepsilon_j\theta
 \geq-c^{-1}\varepsilon_j\omega_j.
\end{equation}

We next explain why this regular-locus estimate gives the same
synthetic estimate on the completion.  Apply Theorem
\ref{thm:main2} to $\omega_j$.  Indeed, if
$\ric(\theta)\leq A\theta$ and $\varepsilon_j\leq1$, then
\begin{equation}
 \ddbar f_j\geq-(A+1)\theta,
\end{equation}
so the $\PSH$ and $L^p$ constants needed there are uniform in $j$.
It follows that the resulting canonical spaces
\begin{equation}
 (X,d_j,\mu_j),\qquad \mu_j= \omega_j^n,
\end{equation}
are $\RCD(-\Lambda,2n)$ for one constant $\Lambda$ independent of
$j$.

We also record why $d_j$ is locally induced by the smooth metric
$\omega_j$ on $U_j$.  Let $\omega_{j,k}$ be the smooth metrics just
used and choose open sets $W'\Subset W\Subset U_j$.  They satisfy
$\omega_{j,k}\geq c\theta$ globally and converge smoothly to
$\omega_j$ on $W$.  Thus every curve from $W'$ to $X\setminus W$
has $\omega_{j,k}$-length at least
\begin{equation}
 \sqrt{c}\,d_\theta(W',X\setminus W)>0,
\end{equation}
whereas sufficiently close points of $V$ can be joined inside $W$
with uniformly smaller length.  Consequently no minimizing curve
between such points leaves $W$.  Passing to the uniform limit of the
distance functions shows that $d_j$ agrees locally with the
Riemannian distance of $\omega_j$.

The analytic set $Z_j$ has zero capacity for this almost-smooth
structure by Lemma \ref{lem:analytic-set-cutoffs}.  Thus there are
logarithmic cutoffs
$\chi_{j,k}\in C_c^\infty(U_j)$ which satisfy
\begin{equation}\label{analytic-set-cutoffs}
 0\leq\chi_{j,k}\leq1,\qquad
 \chi_{j,k}\equiv1\quad\hbox{on every } V\Subset U_j
 \hbox{ for all sufficiently large }k,
\end{equation}
and
\begin{equation}\label{analytic-set-cutoff-energy}
 \int_{U_j}\left(
 |\nabla\chi_{j,k}|_{\omega_j}^2+
 |\Delta_{\omega_j}\chi_{j,k}|
 \right)d\mu_j\longrightarrow0.
\end{equation}
Since the Monge-Amp\`ere measure of a bounded potential does not
charge pluripolar sets, we also have $\mu_j(Z_j)=0$.

The $\RCD(-\Lambda,2n)$ property supplies the
Sobolev-to-Lipschitz property, compactness of
$W^{1,2}\hookrightarrow L^2$, and Lipschitz regularity of every
eigenfunction.  Rescale $\mu_j$ by the fixed positive factor which
identifies its restriction to $U_j$ with Riemannian Hausdorff
measure.  Honda's almost-smooth criterion
\cite[Corollary 3.10]{Honda}, applied with
\eqref{regular-locus-ricci}-\eqref{analytic-set-cutoff-energy},
improves the lower bound to
\begin{equation}\label{analytic-approximation-rcd}
 (X,d_j,\mu_j)\text{ is }\RCD(-c^{-1}\varepsilon_j,2n).
\end{equation}
Returning to $\mu_j= \omega_j^n$ causes no issue, since the
criterion and the RCD condition are invariant under multiplying the
reference measure by a positive constant.
This step is the weak Bochner passage across $Z_j$: the classical
Bochner inequality is used on $U_j$, and the logarithmic cutoff error
vanishes by \eqref{analytic-set-cutoff-energy}.

By \eqref{analytic-singularity-ma} and dominated convergence,
\begin{equation}
 \|\mu_j-\mu\|_{L^1(X)}\longrightarrow0.
\end{equation}
We now justify the current-current form of Corollary
\ref{cor:main1} used here.  Choose diagonal smooth regularizations
$\widehat\omega_j$ of $\omega_j$ and $\widetilde\omega_j$ of
$\omega$ so that their total-variation and uniform distance errors
relative to the corresponding canonical spaces are at most $j^{-1}$.
All these smooth metrics lie in one fixed
$\cK_\theta(p,K';\lambda')$.  Since
$\|\mu_j-\mu\|_{L^1(X)}\to0$, Corollary
\ref{cor:main1}, applied to
$\widehat\omega_j$ and $\widetilde\omega_j$, and then the triangle
inequality give
\begin{equation}\label{analytic-approximation-distance-convergence}
 \sup_{x,y\in X}|d_j(x,y)-d_\omega(x,y)|\longrightarrow0.
\end{equation}
Hence $(X,d_j,\mu_j)$ converges to $(X,d_\omega,\mu)$ in measured
Gromov-Hausdorff topology.  Stability of finite-dimensional RCD
spaces \cite{GMS}, together with
$c^{-1}\varepsilon_j\to0$, proves that the limit is
$\RCD(0,2n)$.   
\end{proof}
 
The following definition for a klt class is a natural extension for klt singularities. 
\begin{definition}
\label{psklt}
Let $X$ be a compact K{\"a}hler manifold and let
$\alpha\in H^{1,1}(X,\R)$ be a pseudoeffective class.  Let
$T_{\min}\in\alpha$ be a positive closed $(1,1)$-current with
minimal singularities, and write
$T_{\min}=\eta+\ddbar\varphi_{\min}$ for a smooth representative
$\eta\in\alpha$.  We say that $\alpha$ is klt if
$e^{-\varphi_{\min}}\in L^1(X)$.
\end{definition}
In particular, the klt definition for $\alpha$ does not depend on the choice of $\eta$.  
 
We can prove Theorem \ref{thm:main2'}.

\begin{corollary}\label{cor:anticanonical-rcd-zero}
Let $(X,\theta)$ be a compact $n$-dimensional K\"ahler manifold.
If $-K_X$ is pseudoeffective and klt, then there exist $p>1$,
$K>0$, $\lambda>0$, and a K\"ahler current
$\omega\in\overline{\cK_\theta(p,K;\lambda)}$ with
non-negative Ricci current such that $(X,\omega)$ induces a unique
$\RCD(0,2n)$ space homeomorphic to $X$.
    
\end{corollary}

\begin{proof}
Since $-K_X$ is pseudoeffective and $\ric(\theta)$ represents
$c_1(X)$, there is a potential $f$ with minimal singularities such
that
\begin{equation}
 f\in\PSH(X,\ric(\theta)).
\end{equation}
The klt assumption and the strong openness theorem \cite{GZ} give
\begin{equation}
 e^{-f}\in L^p(X,\theta^n)
\end{equation}
for some $p>1$.  After adding a constant to $f$, we may assume
\begin{equation}
 \int_Xe^{-f}\theta^n=[\theta]^n.
\end{equation}
Consider the complex Monge-Amp\`ere equation
\begin{equation}
 (\theta+\ddbar\varphi)^n=e^{-f}\theta^n.
\end{equation}
Since $e^{-f}\in L^p(X,\theta^n)$, there exists a unique normalized
solution $\varphi\in\PSH(X,\theta)\cap L^\infty(X)$.  If
$\omega=\theta+\ddbar\varphi$, then
\begin{equation}
 \ric(\omega)=\ric(\theta)+\ddbar f\geq0
\end{equation}
as currents since $f\in\PSH(X,\ric(\theta))$.  The conclusion now
follows from Theorem \ref{thm:current-bochner} by choosing 
$\lambda>0$ with $\ric(\theta)\leq\lambda\theta$ and by taking 
$K\geq\|e^{-f}\|_{L^p(X,\theta^n)}$ to obtain
$\omega\in\overline{\cK_\theta(p,K;\lambda)}$.
    
\end{proof}


\section{Proof of Theorem \ref{thm:main3}}

We begin our proof for Theorem \ref{thm:main3} by writing 
\begin{equation}
 \omega_0=\theta+\ddbar\varphi_0,
 \qquad
 f=-\log\frac{\omega_0^n}{\theta^n},
 \qquad \sup_X\varphi_0=0.
\end{equation}
With our convention
$\ric(\omega)=-\ddbar\log\omega^n$, the background forms for the
unnormalized flow are
\begin{equation}\label{krf-background}
 \theta_t=\theta-t\ric(\theta).
\end{equation}
Thus $[\omega(t)]=[\theta]-tc_1(X)$, and the potential equation has
the sign
\begin{equation}\label{maflow}
 \ddt\varphi
 =\log\frac{(\theta_t+\ddbar\varphi)^n}{\theta^n},
 \qquad \varphi(0)=\varphi_0.
\end{equation}
The existence, uniqueness, and stability of this weak flow follow
from \cite{ST3,DNL}.  We give the additional uniform estimates needed
for the distance assertion.

Use the regularization constructed in the proof of Theorem
\ref{thm:main2}:
\begin{equation}\label{krf-initial-regularization}
 \omega_{0,j}=\theta+\ddbar\varphi_{0,j}>0,
 \qquad
 \omega_{0,j}^n=e^{-f_j+\epsilon_j}\theta^n,
 \qquad \sup_X\varphi_{0,j}=0,
\end{equation}
where
\begin{equation}
 f_j\in C^\infty(X)\cap\PSH(X,\lambda\theta),
 \qquad f_j\searrow f,
 \qquad \epsilon_j\longrightarrow0.
\end{equation}
The estimates already proved there give, uniformly in $j$,
\begin{equation}\label{krf-initial-uniformity}
 \|\varphi_{0,j}\|_{L^\infty(X)}\leq C,
 \qquad
 \omega_{0,j}\geq c_0\theta,
 \qquad
 \ric(\omega_{0,j})\geq-C_0\theta,
\end{equation}
and
\begin{equation}\label{krf-initial-convergence}
 \varphi_{0,j}\longrightarrow\varphi_0\quad\hbox{uniformly},
 \qquad
 \|\omega_{0,j}^n-\omega_0^n\|_{L^1(X)}\longrightarrow0.
\end{equation}

Choose $\tau>0$ so small that
\begin{equation}\label{krf-background-equivalence}
 \frac12\theta\leq\theta_t\leq2\theta
 \qquad(0\leq t\leq\tau).
\end{equation}
Let $\varphi_j$ solve the smooth equation
\begin{equation}\label{maflow2}
 \partial_t\varphi_j
 =\log\frac{(\theta_t+\ddbar\varphi_j)^n}{\theta^n},
 \qquad \varphi_j(0)=\varphi_{0,j},
\end{equation}
and put
\begin{equation}
 \omega_j(t)=\theta_t+\ddbar\varphi_j(t),
 \qquad u_j=\partial_t\varphi_j
 =\log\frac{\omega_j(t)^n}{\theta^n}.
\end{equation}
The stability of the weak K\"ahler-Ricci flow implies that these
solutions converge to \eqref{maflow}, smoothly on
$X\times[a,\tau]$ for each $a>0$.

\begin{lemma}\label{schlem}
After decreasing $\tau$ if necessary, there is a constant $C>0$,
independent of $j$, such that on $X\times[0,\tau]$,
\begin{equation}\label{krf-short-time-estimates}
 |\varphi_j|\leq C,
 \qquad u_j\geq-C,
 \qquad \tr_{\omega_j(t)}\theta\leq C.
\end{equation}
Consequently, $\omega_j(t)\geq C^{-1}\theta$, and the same lower
bound holds for $\omega(t)$ when $t>0$.
\end{lemma}

\begin{proof}
The scalar maximum principle applied to \eqref{maflow2}, using
\eqref{krf-background-equivalence} and the first estimate in
\eqref{krf-initial-uniformity}, gives $|\varphi_j|\leq C$.

Differentiating \eqref{maflow2} yields
\begin{equation}\label{krf-u-evolution}
 (\partial_t-\Delta_{\omega_j})u_j
 =-\tr_{\omega_j}\ric(\theta),
\end{equation}
while
\begin{equation}
 (\partial_t-\Delta_{\omega_j})\varphi_j
 =u_j-n+\tr_{\omega_j}\theta_t.
\end{equation}
Choose $A>0$ such that
$A\theta_t-\ric(\theta)\geq\theta$ on $[0,\tau]$, and set
$H_j=u_j+A\varphi_j$. Then
\begin{align}
 (\partial_t-\Delta_{\omega_j})H_j
 &=Au_j-An+\tr_{\omega_j}
       (A\theta_t-\ric(\theta))\notag\\
 &\geq AH_j-C.\label{krf-H-lower}
\end{align}
At $t=0$,
$u_j(0)=-f_j+\epsilon_j$ is bounded below because
$f_j\leq f_1$.  Thus $H_j(0)$ is uniformly bounded below, and the
parabolic maximum principle applied to \eqref{krf-H-lower} gives
$H_j\geq-C$.  The uniform potential bound then gives $u_j\geq-C$.

Set $S_j=\tr_{\omega_j}\theta$. Applying the parabolic Schwarz inequality to the identity map gives
\begin{equation}
 (\partial_t-\Delta_{\omega_j})\log S_j\leq C_1S_j.
\end{equation}
Since $\theta_t\geq\frac12\theta$, for $B>2C_1$ we obtain
\begin{equation}
 (\partial_t-\Delta_{\omega_j})
   (\log S_j-B\varphi_j)
 \leq-\left(\frac B2-C_1\right)S_j+C.
\end{equation}
The initial value of $S_j$ is uniformly bounded by
\eqref{krf-initial-uniformity}.  Another application of the maximum
principle gives $S_j\leq C$, proving
\eqref{krf-short-time-estimates}.  Passing to the smooth limit for
positive time proves the assertion for $\omega(t)$.
\end{proof}

\begin{lemma}\label{lpflow}
There is a constant $C>0$ such that, as measures on $X$,
\begin{equation}\label{krf-volume-domination}
 \omega(t)^n\leq e^{Ct}\omega_0^n
 \qquad(0<t\leq\tau).
\end{equation}
In particular,
\begin{equation}\label{krf-volume-lp}
 \sup_{0<t\leq\tau}
 \left\|\frac{\omega(t)^n}{\theta^n}\right\|_{L^p(X,\theta^n)}
 <\infty.
\end{equation}
\end{lemma}

\begin{proof}
Let
\begin{equation}
 F_j(t)=\log\frac{\omega_j(t)^n}{\omega_{0,j}^n}
 =u_j(t)-u_j(0).
\end{equation}
Since
$\ddbar u_j(0)=\ric(\theta)-\ric(\omega_{0,j})$,
\eqref{krf-u-evolution} gives
\begin{equation}
 (\partial_t-\Delta_{\omega_j})F_j
 =-\tr_{\omega_j}\ric(\omega_{0,j}).
\end{equation}
Equations \eqref{krf-initial-uniformity} and
\eqref{krf-short-time-estimates} imply that the right hand side is
bounded above by a uniform constant.  Since $F_j(0)=0$, the maximum
principle yields
\begin{equation}
 \omega_j(t)^n\leq e^{Ct}\omega_{0,j}^n.
\end{equation}
Letting $j\to\infty$ and using \eqref{krf-initial-convergence} proves
\eqref{krf-volume-domination}.  The $L^p$ estimate follows at once
from the defining $L^p$ bound for $\omega_0^n/\theta^n$.
\end{proof}

\begin{lemma}\label{l1flow}
As $t\to0^+$,
\begin{equation}\label{krf-volume-tv}
 \|\omega(t)^n-\omega_0^n\|_{L^1(X)}\longrightarrow0.
\end{equation}
\end{lemma}

\begin{proof}
Put $V_t=\int_X\omega(t)^n$. Cohomological invariance of the total
Monge-Amp\`ere mass gives
\begin{equation}
 V_t=\int_X(\theta-t\ric(\theta))^n,
 \qquad |V_t-V_0|\leq Ct.
\end{equation}
By \eqref{krf-volume-domination}, write
$\omega(t)^n=h_t\omega_0^n$ with $0\leq h_t\leq e^{Ct}$. Then
\begin{align}
 \|\omega(t)^n-\omega_0^n\|_{L^1(X)}
 &=\int_X|h_t-1|\,\omega_0^n\notag\\
 &\leq2(e^{Ct}-1)V_0+|V_t-V_0|\leq Ct,
\end{align}
which proves \eqref{krf-volume-tv}.
\end{proof}

We now correct the change-of-class issue. Define
\begin{equation}\label{krf-fixed-class-metric}
 \widehat\omega(t):=\omega(t)+t\ric(\theta)
 =\theta+\ddbar\varphi(t).
\end{equation}
The lower bound in Lemma \ref{schlem} and the smoothness of
$\ric(\theta)$ imply, after decreasing $\tau$, that
\begin{equation}\label{krf-metric-equivalence}
 (1-Ct)\omega(t)\leq\widehat\omega(t)\leq(1+Ct)\omega(t).
\end{equation}
Thus $\widehat\omega(t)$ is a smooth K\"ahler metric in the fixed
class $[\theta]$.  The determinant comparison in
\eqref{krf-metric-equivalence}, together with Lemmas \ref{lpflow}
and \ref{l1flow}, gives
\begin{equation}\label{krf-fixed-class-convergence}
 \sup_{0<t\leq\tau}
 \left\|\frac{\widehat\omega(t)^n}{\theta^n}
 \right\|_{L^p(X,\theta^n)}<\infty,
 \qquad
 \|\widehat\omega(t)^n-\omega_0^n\|_{L^1(X)}\longrightarrow0.
\end{equation}

Fix $\epsilon>0$. Corollary \ref{cor:main2}, applied in the fixed
class $[\theta]$ using \eqref{krf-fixed-class-convergence}, gives
for all sufficiently small $t>0$,
\begin{equation}\label{krf-fixed-class-distance}
 d_{\widehat\omega(t)}(x,y)
 <d_{\omega_0}(x,y)+\frac{\epsilon}{2}
 \qquad(x,y\in X).
\end{equation}
The uniform $L^p$ bound also gives
$\diam(X,\widehat\omega(t))\leq D$.  From
\eqref{krf-metric-equivalence},
\begin{equation}
 d_{\omega(t)}(x,y)
 \leq(1-Ct)^{-1/2}d_{\widehat\omega(t)}(x,y).
\end{equation}
Combining this with \eqref{krf-fixed-class-distance} and then taking
$t$ still smaller gives
\begin{equation}
 d_{\omega(t)}(x,y)-d_{\omega_0}(x,y)
 <\bigl((1-Ct)^{-1/2}-1\bigr)D+\frac{\epsilon}{2}
 <\epsilon
\end{equation}
uniformly in $x,y$. This proves Theorem \ref{thm:main3}.





\begin{thebibliography}{99}

\bibitem{BK}
B\l ocki, Z. and Ko\l odziej, S.,
{\em On regularization of plurisubharmonic functions on manifolds},
Proc. Amer. Math. Soc. 135 (2007), no. 7, 2089-2093

\bibitem{CC1} Cheeger, J. and Colding, T. H. {\em On the structure of spaces with Ricci curvature bounded below, I}, Journal of Differential Geometry, 46(3), 406–480

\bibitem{Dem92} Demailly, J.-P.,
{\em Regularization of closed positive currents and intersection
theory}, J. Algebraic Geom. 1 (1992), no. 3, 361-409

\bibitem{DNL} Di Nezza, E. and Lu, C. H.,
{\em Uniqueness and short time regularity of the weak
K\"ahler-Ricci flow}, Adv. Math. 305 (2017), 953-993


 \bibitem{DZ} Dinew, S. and Zhang, Z. {\em On stability and continuity of bounded solutions of degenerate complex Monge-Amp\`ere equations over compact K\"ahler manifolds}, 
Adv. Math.   225 (2010), no. 1, 367-388

\bibitem{EKS} Erbar, M., Kuwada, K. and Sturm, K.-T.,
{\em On the equivalence of the entropic curvature-dimension condition
and Bochner's inequality on metric measure spaces}, Invent. Math. 201
(2015), no. 3, 993-1071

 \bibitem{DS} Donaldson, S.K. and Sun, S. {\em Gromov-Hausdorff limits of K\"ahler manifolds and algebraic geometry},  Acta Math. 213 (2014), no. 1, 63-106

\bibitem{FGS1}  Fu, X., Guo, B. and  Song, J. {\em Geometric estimates for complex Monge-Amp\`ere equations}, J. Reine Angew. Math. 765 (2020), 69-99


\bibitem{FGS2}  Fu, X., Guo, B. and  Song, J. {\em RCD structures on singular  K\"ahler spaces of complex dimension $3$}, arXiv:2503.08865

\bibitem{FGSW1}  Fu, X., Guo, B., Song, J. and Wang, J. {\em Fundamental groups of compact  K\"ahler varieties with nef anti canonical bundle
}, arXiv:2602.07420

\bibitem{FGSW2}  Fu, X., Guo, B., Song, J. and Wang, J. {\em Fundamental groups of compact  K\"ahler varieties with nef anti canonical bundle II}, in preparation  

\bibitem{GZ} Guan, Q. and Zhou, X. {\em A proof of Demailly's strong openness conjecture}, Annals of Mathematics, 182(2), 605-616

\bibitem{GSS} Guo, B., Song, J. and Sturm, J.
{\em Geometric H\"older estimates for complex Monge-Amp\`ere equations}, in preparation 

\bibitem{GKSS} Guo, B., Ko\l odziej, S., Song, J. and Sturm, J.
{\em Holder estimates for degenerate complex Monge-Amp\`ere equations,}, arXiv:2508.20933


\bibitem{GPSS1} Guo, B., Phong, D.H., Song, J.  and Sturm, J. {\em Sobolev inequalities on K\"ahler spaces}, 2023, arXiv:2311.00221

\bibitem{GPSS2}Guo, B., Phong, D.H., Song, J.  and Sturm, J.  {\em Diameter estimates in K\"ahler geometry},  {Comm. Pure Appl. Math.}, Volume 77, Issue 8 (2024), 3520-3556

\bibitem{GPSS3}Guo, B., Phong, D.H., Song, J.  and Sturm, J.   {\em Diameter estimates in K\"ahler geometry II: removing the small degeneracy assumption}, Math. Z. 308, 43 (2024) 

\bibitem{GMS} Gigli, N., Mondino, A. and Savar\'e, G.,
{\em Convergence of pointed non-compact metric measure spaces and
stability of Ricci curvature bounds and heat flows}, Proc. Lond. Math.
Soc. (3) 111 (2015), no. 5, 1071-1129

\bibitem{Honda} Honda, S.,
{\em Bakry-\'Emery conditions on almost smooth metric measure
spaces}, Anal. Geom. Metr. Spaces 6 (2018), no. 1, 129-145


 
 \bibitem{Kol1} Ko\l dziej, S.,{\em 
The complex Monge-Amp\`ere equation}, 
Acta Math. 180 (1998), no. 1, 69–117

 \bibitem{Kol2} Ko\l odziej, S.,
{\em The Monge-Amp\`ere equation on compact  K\"ahler manifolds}, 
Indiana Univ. Math. J. 52 (2003), no. 3, 667-686
 
\bibitem{Kol3} Ko\l odziej, S.,
{\em H\"older continuity of solutions to the complex Monge-Amp\'ere equation with the right-hand side in $L^p$: the case of compact K\"ahler manifolds. },
Math. Ann. 342 (2008), no. 2, 379–386

\bibitem{Li}
Li, Y. 
{\em  On collapsing Calabi-Yau fibrations},
J. Differential Geom. 117 (2021), no. 3, 451–483

\bibitem{LS} 
 Liu, G. and Sz\'ekelyhidi, G. {\em Gromov-Hausdorff limits of K\"ahler manifolds with Ricci curvature bounded below}, Geom. Funct. Anal. 32 (2022), no. 2, 236-279


\bibitem{So14} Song, J. {\em Riemannian geometry of  K\"ahler-Einstein currents}, 	arXiv:1404.0445

\bibitem{So15} Song, J. {\em Riemannian geometry of K\"ahler-Einstein currents II: an analytic proof of Kawamata's base point free theorem}, arXiv:1409.8374



\bibitem{ST1} Song, J. and Tian, G. {\em The  K\"ahler-Ricci flow on surfaces of positive Kodaira dimension}, Invent. Math. 170 (2007), no. 3, 609-653

\bibitem{ST2} Song, J. and Tian, G. {\em Canonical measures and  K\"ahler-Ricci flow}, J. Amer. Math. Soc. 25 (2012), no. 2, 303-353 

 \bibitem{ST3} Song, J. and Tian, G. {\em The  K\"ahler-Ricci flow through singularities}, Invent. Math. 207 (2017), no. 2, 519-595 

 
\bibitem{Y} Yau, S.T.,
{\em On the Ricci curvature of a compact K\"ahler manifold and the complex MongeAmp\`ere equation. I}, Commun. Pure Appl. Math. 31 (1978), p. 339–411

 
 
 



 

\end{thebibliography}
\end{document}